\documentclass[a4paper,12pt]{article}
\usepackage{amsmath,amssymb,amsthm,graphics,latexsym}
\usepackage{indentfirst}
\usepackage{graphics}
\usepackage{hyperref}  

\title{\Large Rainbow triangles in edge-colored graphs with large minimum color degree}
\author{Bo Ning\thanks{Corresponding author. College of Computer Science, Nankai University, Tianjin 300350, P.R. China. Research supported by National Natural Science Foundation of China grant 12371350 and the Fundamental Research Funds for the Central Universities, Nankai University (No. 63263259). Email: \url{bo.ning@nankai.edu.cn}.}~~and Yuting Tian\footnote{College of Cryptology and Cyber Science, Nankai University, Tianjin 300350, P.R. China. Email: \url{tianyt@mail.nankai.edu.cn}}}
\date{}

\newtheorem{theorem}{Theorem}[section]
\newtheorem{lemma}[theorem]{Lemma}

\newtheorem{proposition}[theorem]{Proposition}
\newtheorem{problem}[theorem]{Problem}

\usepackage{indentfirst}

\newcommand{\dc}{d^c}
\newcommand{\deltac}{\delta^c}
\newcommand{\rt}{\operatorname{rt}}
\newcommand{\CN}{\operatorname{CN}}

\begin{document}
\maketitle

{\small
\noindent{\bfseries Abstract.}
Let $G$ be an edge-colored graph on $n$ vertices, and let $\deltac(G)$ denote its minimum color degree. 
Li and, independently Li, Ning, Xu, and Zhang, proved that every edge-colored graph on $n$ vertices with $\deltac(G) \ge \frac{n+1}{2}$ contains a rainbow triangle. 
Let $\rt(G)$ denote the number of rainbow triangles in $G$, and define 
\[
f(n) = \min\{ \rt(G) : |V(G)| = n,\ \deltac(G) \ge (n+1)/2 \}.
\]
In \cite{LiNingShiZhang2024}, the following open problem was posed: determine all the values of $f(n)$. 
In this paper, we determine $f(n)$ completely: $f(n) = (n^2-1)/8$ for odd $n\geq 3$, $f(n) = \frac{n^2}{4} - 1$ for all even $n \ge 6,$ and $f(4) = 4$. 
This resolves an open problem raised in \cite{LiNingShiZhang2024}.

}

\section{\large Introduction}

Throughout this paper, all graphs are finite, undirected, and simple. An edge-coloring of a graph $G$ is an arbitrary map from $E(G)$ to a set of colors; the coloring is not assumed to be proper. For $v\in V(G)$, let $\CN_G(v)$ denote the set of colors appearing on the edges incident with $v$, and let $\dc_G(v)=|\CN_G(v)|$. The \emph{minimum color degree} is $\deltac(G)=\min_{v\in V(G)}\dc_G(v)$. A subgraph is called \emph{rainbow} if all of its edges receive pairwise distinct colors. Let $\rt(G)$ denote the number of rainbow triangles in $G$, and $t(G)$ denote the number of triangles in the underlying
ordinary graph of $G$. 

Rainbow subgraphs provide colored analogs of classical extremal problems and have been studied extensively; see the survey of Kano and Li~\cite{KanoLi2008}.  
Rainbow triangles, in particular, have been investigated from multiple perspectives, including rainbow versions of Mantel-type theorems, the Caccetta--H\"aggkvist conjecture, and rainbow triangles in restricted colorings~\cite{Aharoni2023, Aharoni2020, Aharoni2019, Balogh2017,He2024}. 
Many natural extensions, such as rainbow cycles, have also attracted considerable attention; see~\cite{Cada2016,Czygrinow2021}.

Li and Wang~\cite{LiWang2012} conjectured that every $n$-vertex edge-colored graph with $\deltac(G)\ge (n+1)/2$ contains a rainbow triangle.  
This conjecture was proved by Li~\cite{Li2013} and, independently, by Li, Ning, Xu, and Zhang~\cite{LiNingXuZhang2014}.  
Moreover, Li et al.~\cite{LiNingXuZhang2014} established a stronger extremal result: for $n\ge 5$, any $n$-vertex edge-colored graph with $\deltac(G)\ge n/2$ contains a rainbow triangle, except when $n$ is even and $G$ is a properly edge-colored copy of $K_{n/2,n/2}$.  
For further results on Ore-type color-degree and color-neighborhood conditions, we refer to~\cite{LiNingZhang2016,FujitaNingXuZhang2019}.

In ordinary graphs, a classical supersaturation phenomenon for triangles is known: Rademacher (see Erd\H{o}s~\cite{Erdos1962}) proved that every $n$-vertex graph with at least $\frac{n^2+1}{4}$ edges contains at least $\frac{n}{2}$ triangles.
Erd\H{o}s conjectured that any $n$-vertex graph with $e(G)\ge\lfloor\frac{n^2}{4}\rfloor+q$ (where $q<\frac{n}{2}$) contains at least $q\lfloor\frac{n}{2}\rfloor$ triangles, which he verified for $q\leq 3$ \cite{Erdos1962} and other partial results \cite{Erdos1962-2}; the full conjecture was proved by Lov\'asz and Simonovits~\cite{LovaszSimonovits1983}. 
If an $n$-vertex graph has minimum degree at least $\frac{n+1}{2}$, then $e(G)\ge\frac{n^2}{4}+\frac{n}{4}$, so by the Lov\'asz--Simonovits theorem such graphs also exhibit triangle supersaturation.

In contrast, the supersaturation behavior of rainbow triangles in edge-colored graphs is quite different.
Li et al.~\cite{LiNingXuZhang2014} studied sufficient conditions in terms of the total color degree and the sum $e(G)+c(G)$, where $c(G)$ denotes the number of colors appearing on $E(G)$.
Extending their work, Fujita et al.~\cite{FujitaNingXuZhang2019} characterized those edge-colored graphs $G$ with $e(G)+c(G)\ge\binom{n+1}{2}-1$ that do not contain rainbow triangle, and posed the problem of finding the analogous threshold for a prescribed number of rainbow triangles.
This was resolved by Ehard and Mohr~\cite{EhardMohr2020}, who proved the following: if an $n$-vertex edge-colored graph $G$ has $m$ edges and $c$ colors and satisfies $m+c\ge\binom{n+1}{2}+k-1$, then $G$ contains at least $k$ rainbow triangles; moreover, the bound is sharp for $n\ge 3k$.
Hence, the condition $e(G)+c(G)$ does not cause supersaturation.

By contrast, the
minimum-color-degree condition gives
rise to supersaturation.
Li, Ning, Shi, and Zhang~\cite{LiNingShiZhang2024} initiated the study of this phenomenon and demonstrated, among other results, that
\[
\rt(G)\ge\frac{1}{6}\,\deltac(G)\bigl(2\deltac(G)-n\bigr)n.
\]
 
In particular, the threshold $\deltac(G)\ge (n+1)/2$ forces $\rt(G)=\Omega(n^2)$.
Define
\[
f(n)=\min\{\rt(G): |V(G)|=n,\ \deltac(G)\ge (n+1)/2\}.
\]

The following problem was listed as an open problem in \cite{LiNingShiZhang2024}.

\begin{problem}[ \rm{\cite[Problem~2,~pp.~757]{LiNingShiZhang2024}}]
Determine all values of $f(n)$.    
\end{problem}

The purpose of this paper is to provide a complete solution to the above problem.
Our main result is the following.

\begin{theorem}\label{thm:main}
Let $f(n)$ be defined as above. Then the following statements hold.
\begin{enumerate}
\item[\textup{(i)}] For every odd $n\geq 3,$
$f(n)=\frac{n^2-1}{8}.$

\item[\textup{(ii)}] For every even $n\ge 6$,
$f(n)=\frac{n^2}{4}-1.$
Moreover, $f(4)=4$.
\end{enumerate}
\end{theorem}

The proof of Theorem \ref{thm:main}(i) combines the lower bound proved in Section \ref{sec:proof-odd} with the following construction result of Li et al. \cite{LiNingShiZhang2024}.

\begin{lemma}[Li et al. {\cite[Proposition 5]{LiNingShiZhang2024}}]\label{thm:li-odd-upper}
For every odd integer $n\ge 3$, there exists an edge-colored graph $G$ on $n$ vertices with $\deltac(G)\ge (n+1)/2$ and
\[
   \rt(G)=\frac{n^2-1}{8}.
\]
Equivalently, $f(n)\le (n^2-1)/8$ for odd $n$.
\end{lemma}

The even case  requires a different extremal construction. For every even \(n\ge 6\),
let the underlying graph
be $G_0=I_{n/2-1}\vee C_{n/2+1}$, that is, the complete join of an independent set of size \(n/2-1\) and a cycle of length
\(n/2+1\). Color the edges cyclically with \(n/2+1\) colors, as in the
construction given in Section~\ref{sec:proof-even}, so that every vertex sees exactly
\(n/2+1\) colors. The only triangles are the triples consisting of one vertex
from \(I_{n/2-1}\) and one edge of \(C_{n/2+1}\), so there are exactly $\left(\frac n2-1\right)\left(\frac n2+1\right)=\frac{n^2}{4}-1$
triangles; with this cyclic coloring all of them are rainbow. This gives the
upper bound $f(n)\le \frac{n^2}{4}-1$ for every even \(n\ge 6\). The exceptional case \(n=4\) is separate:
\(\delta^c(G)\ge 3\) forces the underlying graph to be \(K_4\), and all four
triangles are rainbow.

\noindent
{\underline{\bf {Proof strategy.}}}
Our lower-bound proof treats odd and even $n$ separately.
For odd $n$, we bound the number of
non-rainbow triangles by counting
monochromatic angles, i.e., pairs of the
same colored edges meeting at a
vertex and combine this estimate with the Lov\'asz--Simonovits theorem for small $q$.
The even case is more delicate.
We pass to the complement of the underlying simple graph and reduce the problem to an extremal estimate on $2e(F)+t(F)$ under the constraint $\Delta(F)\le r-2$.
The proof of this estimate relies on several classical and recent tools: Goodman's formula~\cite{Goodman1959}, Chase's theorem on the maximum number of triangles in graphs with given maximum degree~\cite{Chase2020}, Fisher's lower bound for the number of triangles~\cite{Fisher1989}, and a special case of Lo's theorem on cliques in graphs with a given minimum degree~\cite{Lo2012}.

The paper is organized as
follows. In Section \ref{sec:proof-odd}, we present the proof of Theorem \ref{thm:main}(i). Section \ref{sec:proof-even} proves Theorem \ref{thm:main}(ii) via the complement extremal problem and then gives the sharp even construction, including the exception case $n=4$.

\section{\large Proof of Theorem \ref{thm:main}(i)}\label{sec:proof-odd}

To avoid any convention
concerning binomial coefficients whose
upper entry is smaller than the lower
entry, let $p_2(x)=x(x-1)/2$ for every integer $x$. Thus, $p_2(0)=p_2(1)=0$.

\begin{lemma}\label{lem:same-color-pairs}
Let $G$ be an edge-colored graph. Then
\[
   \rt(G)\ge t(G)-\sum_{v\in V(G)} p_2\bigl(d_G(v)-\dc_G(v)+1\bigr).
\]
\end{lemma}

\begin{proof}
Fix a vertex $v$. If $d_G(v)=0$, then there are no incident pairs with the same color at $v$, and the desired local contribution is $0=p_2(1)$. Suppose $d_G(v)>0$. Partition the edges incident with $v$ according to their colors, and let the sizes of the color-classes be $a_1,\ldots,a_k$, respectively, where $k=\dc_G(v)$ and $a_1+\cdots+a_k=d_G(v)$. The number of unordered pairs of incident edges at $v$ having the same color is $\sum_{i=1}^k p_2(a_i)$.

For fixed $d_G(v)$ and $k$, this sum is maximized when one color class is as large as possible and all other color classes have size $1$. Indeed, if two classes have sizes $a,b\ge 2$, then replacing them with classes of sizes $a+b-1$ and $1$ changes the sum by
\[
   p_2(a+b-1)-p_2(a)-p_2(b)=(a-1)(b-1)\ge 0.
\]
Repeating this operation gives
\[
   \sum_{i=1}^k p_2(a_i)
   \le p_2\bigl(d_G(v)-k+1\bigr)
   =p_2\bigl(d_G(v)-\dc_G(v)+1\bigr).
\]

Observe that the number of rainbow triangles equals the number of triangles minus the number of non-rainbow triangles in $G$.
Every non-rainbow
triangle has two edges of the same
color. Such triangles are of two types:
those using exactly two colors and
those using only one color.

These two edges meet at a vertex of the triangle, so the triangle is witnessed by an incident pair with the same color. The union bound over all vertices proves the lemma.
\end{proof}

We shall use the following Lov\'asz--Simonovits theorem. 

\begin{lemma}[Lov\'asz--Simonovits {\cite{LovaszSimonovits1983}}]\label{lem:lovasz-Simonovits}
Let $n\ge 3$, and let $q$ be an integer with $0\le q<n/2$. Every graph on $n$ vertices with $\lfloor n^2/4\rfloor+k$ edges contains at least $k\lfloor n/2\rfloor$ triangles.
\end{lemma}

\begin{lemma}\label{lem:ordinary-surplus}
Let $H$ be a graph on $n\ge 3$ vertices and put $k=\lfloor n/2\rfloor+1$. If $\delta(H)\ge k$, then
\[
   t(H)\ge \frac{n^2-1}{8}+\sum_{v\in V(H)}p_2\bigl(d_H(v)-k+1\bigr).
\]
\end{lemma}

\begin{proof}
For an edge $xy\in E(H)$, we have $|N_H(x)\cap N_H(y)|\ge d_H(x)+d_H(y)-n$. Summing this inequality over all edges gives
\begin{equation}\label{eq:degree-triangle}
   3t(H)\ge \sum_{v\in V(H)} d_H(v)^2-n|E(H)|.
\end{equation}

First, suppose that $n=2m$. Then $k=m+1$. Write $x_v=d_H(v)-k\ge 0$, $X=\sum_v x_v$, and $Y=\sum_v x_v^2$. Since $d_H(v)\le 2m-1$, we have $x_v\le m-2$, and therefore $Y\le (m-2)X$. Also $|E(H)|=m(m+1)+X/2$ and
\[
   \sum_v d_H(v)^2=2m(m+1)^2+2(m+1)X+Y.
\]
Substituting these identities into \eqref{eq:degree-triangle} gives
\[
   t(H)\ge \frac{2m(m+1)+(m+2)X+Y}{3}.
\]
Therefore,
\begin{align*}
&t(H)-\sum_v p_2(x_v+1)-\frac{(2m)^2-1}{8}\ge
\frac{4m^2+16m+3}{24}+\frac{(2m+1)X-Y}{6}\ge 0.
\end{align*}
This proves the lemma for $n$ is even.

Now suppose that $n=2m+1$. Again, $k=m+1$, and set $x_v=d_H(v)-k\ge 0$, $X=\sum_v x_v$, and $Y=\sum_v x_v^2$. Here $x_v\le m-1$, so $Y\le (m-1)X$.

If $X\ge m-1$, then \eqref{eq:degree-triangle} yields
\begin{align*}
t(H)-\sum_v p_2(x_v+1)-\frac{m(m+1)}2&\ge
\frac{2mX-Y-(m^2-1)}6\\
&\ge \frac{(m+1)X-(m^2-1)}6\ge 0.
\end{align*}
It remains to consider $X\le m-1$. Since $|E(H)|=m(m+1)+q=\lfloor\frac{n^2}{4}\rfloor+q$ with $q=(m+1+X)/2$, the number $q$ is an integer, and $0\le q\le m<(2m+1)/2,$ Lemma \ref{lem:lovasz-Simonovits} gives
\[
   t(H)\ge qm=\frac{m(m+1)}2+\frac{mX}{2}.
\]
On the other hand, we have $\sum_v p_2(x_v+1)=(X+Y)/2\le mX/2$. Hence
\[
   t(H)\ge \frac{m(m+1)}2+\sum_v p_2(x_v+1),
\]
which is the desired inequality for odd $n$. The proof is complete.
\end{proof}

\begin{proof}[Proof of Theorem \ref{thm:main}\textup{(i)}]
Let $G$ be an edge-colored graph on $n\ge 3$ vertices with $\deltac(G)\ge (n+1)/2$, and put $k=\lfloor n/2\rfloor+1$. Since color degrees are integers, $\dc_G(v)\ge k$ for every vertex $v$. Hence, the underlying ordinary graph of $G$ has minimum degree at least $k$. By Lemma \ref{lem:ordinary-surplus},
\[
   t(G)\ge \frac{n^2-1}{8}+\sum_v p_2(d_G(v)-k+1).
\]
By Lemma \ref{lem:same-color-pairs},
\[
   \rt(G)\ge t(G)-\sum_v p_2(d_G(v)-\dc_G(v)+1).
\]
Since $\dc_G(v)\ge k$ and $d_G(v)\ge \dc_G(v)$, we have
\[
   1\le d_G(v)-\dc_G(v)+1\le d_G(v)-k+1.
\]
The function $p_2(z)=z(z-1)/2$ is increasing for positive integers $z$, so the two preceding inequalities imply $\rt(G)\ge (n^2-1)/8$ for every $n\ge 3$. If $n$ is odd, Lemma \ref{thm:li-odd-upper} gives the reverse inequality. Therefore, $f(n)=(n^2-1)/8$ for every odd $n$.
\end{proof}

\section{\large Proof of Theorem \ref{thm:main}(ii)}\label{sec:proof-even}

In this section, we present the proof of Theorem \ref{thm:main}(ii).
We first record Goodman's formula.

\begin{lemma}[Goodman {\cite{Goodman1959}}]\label{lem:goodman}
Let $F$ be a graph on $n$ vertices, and let $\overline F$ be its complement. Then
\[
  t(F)+t(\overline F)
  =\binom{n}{3}-\frac12\sum_{v\in V(F)}d_F(v)\bigl(N-1-d_F(v)\bigr).
\]
Equivalently,
\[
  t(F)+t(\overline F)
  =\frac12\left(
      \sum_v\binom{d_F(v)}{2}
      +\sum_v\binom{n-1-d_F(v)}{2}
      -\binom n3
    \right).
\]
\end{lemma}

Chase \cite{Chase2020} proved the following theorem, which resolved a conjecture of Gan, Loh, and Sudakov \cite{GanLohSudakov2015}.

\begin{lemma}[Chase {\cite[Theorem 1]{Chase2020}}]\label{lem:chase}
Let $G$ be a graph on $N$ vertices with maximum degree $D$. Let $N=q(D+1)+r$, where $0\le r\leq D$. Then
\[
  t(G)\le q\binom{D+1}{3}+\binom r3.
\]
\end{lemma}

We also need the following variant of Chase's result.

\begin{lemma}\label{lem:chase-revised}
Let $G$ be a graph on $N$ vertices with maximum degree at most $D$. Let $n=q(D+1)+r$, where $0\le r\le D$. Then
\[
  t(H)\le q\binom{D+1}{3}+\binom r3.
\]
\end{lemma}
\begin{proof}
If $\Delta(G)=D$, the bound follows directly from Lemma~\ref{lem:chase}. 
Thus, we assume $\Delta(G)<D$.
If $q=0$, then $N=r\le D$ and the right-hand side equals $\binom{r}{3}$. 
So $t(G)\le\binom{r}{3}$ holds trivially.
We may therefore suppose $q\ge 1$, which gives $n\ge D+1$. 
Let $v$ be a vertex of maximum degree in $G$. 
Because $n\ge D+1$, we can add $D-\deg_G(v)$ edges incident to $v$ (and to other vertices if needed) to obtain a supergraph $\Gamma$ on the same vertex set with $\Delta(\Gamma)=D$. 
Adding edges cannot decrease the number of triangles, so $t(G)\le t(\Gamma)$. 
Applying Lemma~\ref{lem:chase} to $\Gamma$ yields $t(\Gamma)\le q\binom{D+1}{3}+\binom{r}{3}$, and the proof is complete.
\end{proof}

We use Fisher’s
triangle lower bound in the following 
form.

\begin{lemma}[Fisher {\cite{Fisher1989}}]\label{lem:fisher}
Let $G$ be a graph with $N$ vertices and $E$ edges. If $N^2/4\le E\le N^2/3$, then
\[
  t(G)\ge
  \frac{9EN-2N^3-2(N^2-3E)^{3/2}}{27}.
\]
\end{lemma}

Finally, we need the following special case of Lo's minimum-degree theorem for triangles.

\begin{lemma}[Lo \rm {\cite[Theorem~2.1]{Lo2012}}]\label{lem:lo}
Let $G$ be an $n$-vertex graph with minimum degree at least $\delta$ such that 
$\frac{n}{2} < \delta \le  \frac{2n}{3}$. Then $$t(G)\geq\frac12(1-2\beta)(1-\beta)\beta n^3,$$
where $\beta=1-\frac{\delta}{n}$.
\end{lemma}

Using the above lemma, we can prove the following useful lemma.

\begin{lemma}\label{lem:lo-special}
Let $m\ge 3$. Every graph on $2m$ vertices with minimum degree at least $m+1$ contains at least $m^2-1$ triangles.
\end{lemma}

\begin{proof}
Let $G$ be a graph on $2m$ vertices with minimum degree at least $m+1$. In Lo's notation, put $\beta=1-(m+1)/(2m)=(m-1)/(2m)$. Then $1-2\beta=1/m$ and $1-\beta=(m+1)/(2m)$. Since $m<m+1\le 4m/3$ for $m\ge 3$, Lemma \ref{lem:lo} applies with $n=2m$ and $\delta=m+1$. Therefore,
\[
  t(H)\ge \frac12(1-2\beta)(1-\beta)\beta(2m)^3=m^2-1.
\]
\end{proof}

We next convert the colored problem into a complement problem.

\begin{lemma}\label{lem:colored-reduction}
Let $G$ be an edge-colored graph on $2m\ge 6$ vertices with $\deltac(G)\ge m+1$. Let $G_0$ be its underlying ordinary graph, and put $F=\overline{G_0}$ and $x_v=d_F(v)$. Then $\Delta(F)\le m-2$, and
\[
  \rt(G)\ge
  \frac{m(m-1)(m+4)}{3}-\bigl(2e(F)+t(F)\bigr).
\]
\end{lemma}

\begin{proof}
Since $\dc_G(v)\ge m+1$, we have $d_{G_0}(v)\geq m+1$. Thus, $x_v=d_F(v)=2m-1-d_{G_0}(v)\le m-2$, and so $\Delta(F)\le m-2$.

At a vertex $v$, the graph $G_0$ has $2m-1-x_v$ incident edges, and these edges use at least $m+1$ colors. If $d$ objects are placed into at least $q$ nonempty classes, then the number of unordered pairs lying in a common class is at most $\binom{d-q+1}{2}$. Hence, the number of same-color unordered pairs of incident edges at $v$ is at most $\binom{m-1-x_v}{2}$. Every non-rainbow triangle
contains a pair of incident edges of the
same color. Counting
all such same-color incident pairs over all vertices may count a non-rainbow triangle more
than once, but this gives a valid upper bound on the number of non-rainbow triangles. So
\[
  \rt(G)
  \ge
  t(G_0)-\sum_v\binom{m-1-x_v}{2}.
\]
By inclusion-exclusion, equivalently, by Lemma \ref{lem:goodman},
\[
  t(G_0)
  =\binom{2m}{3}-(2m-2)e(F)+\sum_v\binom{x_v}{2}-t(F).
\]
Substituting this identity and using
\[
  \binom{x}{2}-\binom{m-1-x}{2}
  =(m-2)x-\binom{m-1}{2},
  \qquad
  \sum_vx_v=2e(F),
\]
we obtain
\begin{align*}
  \rt(G)
  &\ge
  \binom{2m}{3}-(2m-2)e(F)
  +2(m-2)e(F)-2m\binom{m-1}{2}-t(F) \\
  &=\frac{m(m-1)(m+4)}{3}-\bigl(2e(F)+t(F)\bigr).
\end{align*}
\end{proof}

The required lower bound is reduced to the following ordinary graph estimate.

\begin{lemma}\label{lem:complement-estimate}
Let $m\ge 3$, and let $F$ be a graph on $2m$ vertices with $\Delta(F)\le m-2$. Then
\[
  2e(F)+t(F)\le \frac{m^3-4m+3}{3}.
\]
\end{lemma}

Before proving Lemma \ref{lem:complement-estimate}, we isolate the algebraic estimate needed after applying Fisher's theorem.

\begin{lemma}\label{lem:algebra}
Let $m\ge 6$ and $1\le s\le (m-3)/2$. Define
\[
  L_m(s)=
  \frac{
    18m\bigl(m(m+1)+s\bigr)-16m^3
    -2\bigl(m^2-3m-3s\bigr)^{3/2}
  }{27}.
\]
Then $L_m(s)\ge m^2+2s^2+s-1$.
\end{lemma}

\begin{proof}
Consider $s$ as a real variable and set $\Phi_m(s)=L_m(s)-(m^2+2s^2+s-1)$. On the interval $1\le s\le (m-3)/2$, we have $m^2-3m-3s\ge (m-3)(2m-3)/2>0$. A direct differentiation gives
\[
  \Phi_m''(s)=-4-\frac{1}{2\sqrt{m^2-3m-3s}}<0.
\]
Thus, $\Phi_m$ is concave in this interval. Since $\Phi_m$ is concave, for every $s$ in the interval the value $\Phi_m(s)$ is at least the linear interpolation of its two endpoint values. Thus, it is enough to prove $\Phi_m(1)\geq 0$ and $\Phi_m((m-3)/2)\geq 0$.

For $s=1$,
\[
  27\Phi_m(1)
  =2m^3-9m^2+18m-54
  -2(m^2-3m-3)^{3/2}.
\]
It suffices to prove
\[
  (m^2-3m-3)^{3/2}
  \le
  \frac{2m^3-9m^2+18m-54}{2}.
\]
The right-hand side is positive for $m\ge 6$. After squaring, the difference between the square of the right-hand side and the square of the left-hand side is
\[
  \frac{27}{4}(m-4)(3m^3-12m^2+8m-28)>0.
\]
Indeed, $3m^3-12m^2+8m-28$ is positive at $m=6$, and its derivative $9m^2-24m+8$ is positive for $m\ge 6$. Hence $\Phi_m(1)\ge 0$.

For $s=(m-3)/2$, it suffices to prove
\[
  \left(\frac{(m-3)(2m-3)}2\right)^{3/2}
  \le
  \frac{(m-3)(4m^2-15m+36)}4.
\]
The right-hand side is positive for $m\ge 6$. After squaring, the difference between the square of the right-hand side and the square of the left-hand side is
\[
  \frac{27}{16}(m-3)^2(7m^2-26m+42)>0.
\]
Since the quadratic
$7m^2-26m+42$ has negative
discriminant and positive leading
coefficient, it is positive for all real $m$. Therefore, the second endpoint is also non-negative, and the lemma follows by concavity.
\end{proof}

\begin{proof}[Proof of Lemma \ref{lem:complement-estimate}]
Put $b_m=(m^3-4m+3)/3$. If $m=3$, then $\Delta(F)\le 1$, so $t(F)=0$ and $e(F)\le 3$. Hence $2e(F)+t(F)\le 6=b_3$. We henceforth assume that $m\ge 4$.

Let $E_0=(m-1)(2m-3)/2$. First, suppose that $e(F)\le E_0$. Apply Lemma \ref{lem:chase-revised} with $N=2m$ and $D=m-2$. Since $m\ge 4$, we have $2m=2(m-1)+2$, so $q=2$ and $r=2$. Therefore $t(F)\le 2\binom{m-1}{3}$, and consequently
\begin{align*}
  2e(F)+t(F)
  &\le 2E_0+2\binom{m-1}{3} \\
  &=(m-1)(2m-3)+\frac{(m-1)(m-2)(m-3)}3=b_m.
\end{align*}

It remains to consider $e(F)>E_0$. Since $\Delta(F)\le m-2$, the maximum possible number of edges is $m(m-2)$. Write $e(F)=m(m-2)-s$, where $s\ge 0$ is an integer. The inequality $e(F)>E_0$ gives
\begin{equation}\label{eq:s-range}
  s<\frac{m-3}{2}.
\end{equation}

If $s=0$, then $e(F)=m(m-2)$, so every vertex of $F$ has degree exactly $m-2$. Put $G_0=\overline F$. Then $G_0$ is an $(m+1)$-regular graph on $2m$ vertices. Lemma \ref{lem:lo-special} gives $T(G_0)\ge m^2-1$. By Lemma \ref{lem:goodman},
\[
  t(F)+t(G_0)=\binom{2m}{3}-m(m-2)(m+1),
\]
where the second term follows from
\[
  \frac12\sum_vd_F(v)(2m-1-d_F(v))=m(m-2)(m+1).
\]
Therefore
\[
  2e(F)+t(F)
  \le 2m(m-2)+\binom{2m}{3}
       -m(m-2)(m+1)-(m^2-1)=b_m.
\]

Finally, suppose that $s\ge 1$. Since $s$ is an integer, \eqref{eq:s-range} implies $m\ge 6$. Let $y_v=(m-2)-d_F(v)\ge 0$. Since $e(F)=m(m-2)-s$, we have $\sum_v y_v=2s$, and hence
\begin{equation}\label{eq:y-square}
  \sum_v y_v^2\le \left(\sum_vy_v\right)^2=4s^2.
\end{equation}
Moreover, $d_F(v)=m-2-y_v$ and $2m-1-d_F(v)=m+1+y_v$. Thus
\[
  d_F(v)(2m-1-d_F(v))
  =(m-2)(m+1)-3y_v-y_v^2.
\]
Summing over all vertices and using \eqref{eq:y-square} gives
\begin{equation}\label{eq:goodman-degree-lower}
\begin{aligned}
  \sum_vd_F(v)(2m-1-d_F(v))
  &=2m(m-2)(m+1)-6s-\sum_vy_v^2 \\
  &\ge 2m(m-2)(m+1)-6s-4s^2.
\end{aligned}
\end{equation}

Put $G_0=\overline F$. Then $e(G_0)=\binom{2m}{2}-e(F)=m(m+1)+s$. In the range $1\le s<(m-3)/2$, we have
\[
  m^2\le m(m+1)+s\le \frac{4m^2}{3}.
\]
The lower bound is
immediate, and the upper bound
follows from
\[
  m(m+1)+s<m^2+m+\frac{m-3}{2}
  =\frac{2m^2+3m-3}{2}\le \frac{4m^2}{3}
\]
for $m\ge 6$. Therefore, Lemma \ref{lem:fisher} applies to $G_0$ with $N=2m$ and $E=m(m+1)+s$. It gives $T(G_0)\ge L_m(s)$, and Lemma \ref{lem:algebra} yields
\begin{equation}\label{eq:fisher-processed}
  T(G_0)\ge m^2+2s^2+s-1.
\end{equation}
Using Goodman's formula again,
\[
  t(F)=\binom{2m}{3}
    -\frac12\sum_vd_F(v)(2m-1-d_F(v))-T(G_0).
\]
Combining this identity with \eqref{eq:goodman-degree-lower} and \eqref{eq:fisher-processed}, we get
\[
  t(F)
  \le \binom{2m}{3}
    -\frac12\bigl(2m(m-2)(m+1)-6s-4s^2\bigr)
    -(m^2+2s^2+s-1).
\]
Consequently
\begin{align*}
  2e(F)+t(F)
  &\le 2(m(m-2)-s)+\binom{2m}{3}-m(m-2)(m+1) \\
  &\qquad +3s+2s^2-(m^2+2s^2+s-1) \\
  &=2m(m-2)+\binom{2m}{3}-m(m-2)(m+1)-(m^2-1) \\
  &=b_m.
\end{align*}
This completes the proof.
\end{proof}

\begin{proof}[Proof of Theorem \ref{thm:main}\textup{(ii)} for $n\ge 6$]
Let $G$ be an edge-colored graph on $2r$ vertices with $\deltac(G)\ge r+1$. Apply Lemma \ref{lem:colored-reduction} with $m=r$. For the complement $F$ of the underlying graph, we have $\Delta(F)\le r-2$, and so Lemma \ref{lem:complement-estimate} gives
\[
  2e(F)+t(F)\le \frac{r^3-4r+3}{3}.
\]
Therefore
\[
  \rt(G)
  \ge \frac{r(r-1)(r+4)}{3}
        -\frac{r^3-4r+3}{3}
  =r^2-1.
\]
This proves $f(2r)\ge r^2-1$.

To show sharpness, let the
underlying graph be $G_0=I_{r-1}\vee C_{r+1}$, the complete join of an independent set of size $r-1$ and a cycle of length $r+1$. Write $I_{r-1}=\{a_0,a_1,\ldots,a_{r-2}\}$ and write the cycle as $b_0b_1\cdots b_r b_0$. Use the color set $\mathbb Z_{r+1}$, and define
\[
  c(a_i b_j)=i+j\pmod{r+1}
\]
for $0\le i\le r-2$ and $0\le j\le r$, and define
\[
  c(b_jb_{j+1})=j+r\pmod{r+1}
\]
for $j\in\mathbb Z_{r+1}$.

Each vertex $a_i$ sees all $r+1$ colors on its edges to the cycle. Each vertex $b_j$ sees the colors $j,j+1,\ldots,j+r-2\pmod{r+1}$ on its edges to the independent set, and it sees the two cycle-edge colors $j+r-1$ and $j+r\pmod{r+1}$. Thus every vertex sees exactly $r+1$ colors, and $\deltac(G)=r+1$.

The only triangles in $G_0$ are the triples $a_i b_j b_{j+1}$. There are $(r-1)(r+1)=r^2-1$ such triangles. Their edge colors are $i+j$, $i+j+1$, and $j+r\pmod{r+1}$. These three colors are pairwise distinct because $0\le i\le r-2$. Hence all triangles are rainbow, and the construction has exactly $r^2-1$ rainbow triangles. Therefore $f(2r)\le r^2-1$.
\end{proof}

\begin{proposition}\label{prop:f4}
We have $f(4)=4$.
\end{proposition}

\begin{proof}
If $|V(G)|=4$ and $\deltac(G)\ge 3$, then the degree of each vertex is at least $3$, and the underlying graph is $K_4$. At each vertex, the three incident edges have pairwise distinct colors. In any triangle of this $K_4$, each pair of triangle edges meet at a vertex, and therefore each such pair has distinct colors. Hence, all four triangles are rainbow, so $\rt(G)=4$.
\end{proof}

Together with the preceding proof for $n\ge 6$, Proposition \ref{prop:f4} completes the proof of Theorem \ref{thm:main}(ii).

Theorem \ref{thm:main} determines the threshold problem exactly. For the special case $\deltac(G)\ge (n+1)/2$,
Theorem \ref{thm:main}(ii) improves the general bound of Li et al. \cite{LiNingShiZhang2024} as follows.
For $n=2r$, their estimate gives $\rt(G)\ge 2r(r+1)/3$, whereas Theorem \ref{thm:main}(ii) gives the exact value $r^2-1$ for $r\ge 3$. For $n=2r+1$, their estimate gives $\rt(G)\ge (r+1)(2r+1)/6$, whereas Theorem \ref{thm:main}(i) gives the exact value $r(r+1)/2$.

\section*{AI Statement}
The authors used the AI tool (Chatgpt pro.\,5.5) for proofreading and DeepSeek for grammar revision of the entire paper.

\end{document}